\documentclass[11pt]{article}
\usepackage{amsmath,amsfonts,amssymb,amsthm}
\usepackage{makeidx}  
\usepackage{graphicx}
\usepackage{amsmath}
\usepackage{amsthm}
\usepackage{amssymb}
\usepackage{amsfonts}
\usepackage{amscd}
\usepackage[matrix,arrow,curve]{xy}
\usepackage{mathabx}

\newtheorem{theorem}{Theorem}
\newtheorem{lemma}[theorem]{Lemma}

\newtheorem{prop}[theorem]{Proposition}
\newtheorem{corollary}[theorem]{Corollary}

\theoremstyle{remark}
\newtheorem{remark}{Remark}
\newtheorem{example}[remark]{Example}
\newtheorem{question}[remark]{Question}

\theoremstyle{definition}
\newtheorem*{acknowledgement}{Acknowledgement}

\begin{document}
\author{Denis I.~Saveliev
\footnote{
Russian Academy of Sciences, 
Steklov Mathematical Institute and 
Institute for Information Transmission Problems.
The work was partially supported by grant 
16-11-10252 of the Russian Science Foundation.}
}
\title{On first-order expressibility 
of satisfiability in submodels}
\date{\normalsize{2016, revised June 2019}}
\maketitle


\theoremstyle{plain}
\newtheorem*{tm}{Theorem}
\newtheorem*{lm}{Lemma}
\newtheorem*{prp}{Proposition}
\newtheorem*{cor}{Corollary}

\theoremstyle{definition}
\newtheorem*{df}{Definition}
\newtheorem*{exm}{Example}
\newtheorem*{mdl}{Model} 
\newtheorem*{prb}{Problem}

\theoremstyle{remark}
\newtheorem*{pf}{Proof}
\newtheorem*{rmk}{Remark}
\newtheorem*{hrmk}{Historical remarks}
\newtheorem*{wrn}{Warning}
\newtheorem*{q}{Question}
\newtheorem*{qs}{Questions}
\newtheorem*{ex}{Example}
\newtheorem*{exs}{Examples}

\newcommand{\dom}{ {\mathop{\mathrm {dom\,}}\nolimits} }
\newcommand{\ran}{ {\mathop{\mathrm{ran\,}}\nolimits} }

\newcommand{\id}{ {\mathop{\mathrm {id}}\nolimits} }
\newcommand{\nat}{ {\mathop{\mathrm {nat}}\nolimits} }

\newcommand{\cf}{ {\mathop{\mathrm {cf\,}}\nolimits} }
\newcommand{\cl}{ {\mathop{\mathrm {cl\,}}\nolimits} }
\newcommand{\cof}{ {\mathop{\mathrm {cof\,}}\nolimits} }
\newcommand{\add}{ {\mathop{\mathrm {add\,}}\nolimits} }
\newcommand{\sat}{ {\mathop{\mathrm {sat\,}}\nolimits} }
\newcommand{\tc}{ {\mathop{\mathrm {tc\,}}\nolimits} }
\newcommand{\unif}{ {\mathop{\mathrm {unif\,}}\nolimits} }
\newcommand{\fp}{ {\mathop{\mathrm {fp\,}}\nolimits} }
\newcommand{\fs}{ {\mathop{\mathrm {fs\,}}\nolimits} }
\newcommand{\pr}{ {\mathop{\mathrm {pr\/}}\nolimits} }
\newcommand{\uhr}{\!\upharpoonright\!}
\newcommand{\lra}{ {\leftrightarrow} }
\newcommand{\ot}{ {\mathop{\mathrm {ot\,}}\nolimits} }
\newcommand{\ol}{\overline}
\newcommand{\cnc}{ {^\frown} }
\newcommand{\image}{\/``\,}
\newcommand{\scc}{\beta\!\!\!\!\beta}
\newcommand{\wt}{\widetilde}

\newcommand\bigforall{\mbox{\Large$\mathsurround0pt\forall$}}
\newcommand\Bigforall{\mbox{\LARGE$\mathsurround0pt\forall$}}
\newcommand\bigexists{\mbox{\Large$\mathsurround0pt\exists$}}
\newcommand\Bigexists{\mbox{\LARGE$\mathsurround0pt\exists$}}

\newcommand{\J}{ {\mathrm J} }
\newcommand{\PC}{ {\mathrm {PC}} }
\newcommand{\PRA}{ {\mathrm {PRA}} }
\newcommand{\T}{ {\mathrm T} }
\newcommand{\TA}{ {\mathrm {TA}} }
\newcommand{\PA}{ {\mathrm {PA}} }
\newcommand{\KP}{ {\mathrm {KP}} }
\newcommand{\Z}{ {\mathrm Z} }
\newcommand{\ZF}{ {\mathrm {ZF}} }
\newcommand{\ZFA}{ {\mathrm {ZFA}} }
\newcommand{\ZFC}{ {\mathrm {ZFC}} }
\newcommand{\AEx}{ {\mathrm {AE}} }
\newcommand{\AR}{ {\mathrm {AR}} }
\newcommand{\WR}{ {\mathrm {WR}} }
\newcommand{\AF}{ {\mathrm {AF}} }
\newcommand{\AC}{ {\mathrm {AC}} }
\newcommand{\GC}{ {\mathrm {GC}} }
\newcommand{\DC}{ {\mathrm {DC}} }
\newcommand{\AD}{ {\mathrm {AD}} }
\newcommand{\AInf}{ {\mathrm {AInf}} }
\newcommand{\AU}{ {\mathrm {AU}} }
\newcommand{\AP}{ {\mathrm {AP}} }
\newcommand{\PI}{ {\mathrm {PI}} }
\newcommand{\CH}{ {\mathrm {CH}} }
\newcommand{\GCH}{ {\mathrm {GCH}} }
\newcommand{\APr}{ {\mathrm {APr}} }
\newcommand{\ASp}{ {\mathrm {ASp}} }
\newcommand{\ARp}{ {\mathrm {ARp}} }
\newcommand{\AFA}{ {\mathrm {AFA}} }
\newcommand{\BAFA}{ {\mathrm {BAFA}} }
\newcommand{\FAFA}{ {\mathrm {FAFA}} }
\newcommand{\SAFA}{ {\mathrm {SAFA}} }
\newcommand{\nonempty}{ {\mathrm {nonempty}} }
\newcommand{\I}{ {\mathrm {I}} }
\newcommand{\II}{ {\mathrm {II}} }
\newcommand{\lex}{ {\mathrm {lex}} }

\begin{abstract}
Let $\kappa,\lambda$ be regular cardinals, 
$\lambda\le\kappa$, let $\varphi$ be a~sentence of 
the language $\mathcal L_{\kappa,\lambda}$ in a~given 
signature, and let $\vartheta(\varphi)$ express 
the fact that $\varphi$~holds in a~submodel, i.e., 
any model~$\mathfrak A$ in the signature satisfies 
$\vartheta(\varphi)$ if and only if some 
submodel~$\mathfrak B$ of~$\mathfrak A$ satisfies~$\varphi$. 
It was shown in~\cite{Saveliev Shapirovsky 2018} that,  
whenever $\varphi$~is in $\mathcal L_{\kappa,\omega}$ in 
the signature having less than~$\kappa$ functional symbols 
(and arbitrarily many predicate symbols), then 
$\vartheta(\varphi)$ is equivalent to a~monadic 
existential sentence in the second-order language
$\mathcal L^{2}_{\kappa,\omega}$, and that for any 
signature having at least one binary predicate symbol 
there exists $\varphi$ in $\mathcal L_{\omega,\omega}$ 
such that $\vartheta(\varphi)$ is not equivalent to any
(first-order) sentence in~$\mathcal L_{\infty,\omega}$. 
Nevertheless, in certain cases $\vartheta(\varphi)$ are 
first-order expressible. In this note, we provide several 
(syntactical and semantical) characterizations of the case 
when $\vartheta(\varphi)$ is in $\mathcal L_{\kappa,\kappa}$ 
and $\kappa$ is $\omega$ or a~certain large cardinal.
\end{abstract}


Given a~model-theoretic language~$\mathcal L$ 
(in sense of~\cite{Barwise Feferman}) and 
a~sentence~$\varphi$ in~$\mathcal L$, let 
$\vartheta(\varphi)$ express the fact that 
$\varphi$~is satisfied in a~submodel. 
Thus for any model~$\mathfrak A$,
$$
\mathfrak A\vDash\vartheta(\varphi)
\;\text{ iff }\;
\mathfrak B\vDash\varphi
\text{ for some submodel~$\mathfrak B$ of }\mathfrak A.
$$ 
We study when $\vartheta(\varphi)$, considered a~priori 
as a~meta-expression, is equivalent to a~sentence in 
another (perhaps, the same) given model-theoretic 
language~$\mathcal L'$. Such questions naturally arise in 
studies of modal logics of submodels; if $\mathcal L$~is 
closed under~$\vartheta$, then $\vartheta$~induces a~modal 
operator on sentences (where a~possibility of~$\varphi$ 
means the satisfiability of~$\varphi$ in a~submodel),
and the resulting modal logic can be regarded as 
a~fragment of~$\mathcal L$ with the submodel relation on 
a~given class of models. These logics are an instance 
of modal logics of various model-theoretic relations, 
which were introduced and studied 
in~\cite{Saveliev Shapirovsky 2018}; another instance is 
modal logic of forcing (see, e.g.,~\cite{Hamkins Lowe}). 
As another source of motivation for studies undertaken 
in this note, let us point out the paper~\cite{Montague}
discussing reduction of higher-order logics to 
second-order one.

Here we concentrate on first-order languages 
$\mathcal L_{\kappa,\lambda}$. Recall that 
$\mathcal L_{\omega,\omega}$ is the usual first-order 
finitary language; $\mathcal L_{\kappa,\lambda}$ expands 
it by involving Boolean connectives of any arities~$<\kappa$ 
and quantifiers over~$<\lambda$ first-order variables, 
where $\lambda\le\kappa$ are given regular cardinals; 
and $\mathcal L_{\infty,\lambda}$~is the union 
of $\mathcal L_{\kappa,\lambda}$ for all~$\kappa$;
see \cite{Barwise Feferman}, \cite{Drake},~\cite{Kanamori}. 
It was shown in~\cite{Saveliev Shapirovsky 2018} that, 
even for $\varphi$ in~$\mathcal L_{\omega,\omega}$,
it is possible that $\vartheta(\varphi)$ is not 
in~$\mathcal L_{\infty,\omega}$; on the other hand, 
$\vartheta(\varphi)$~is equivalent to a~second-order 
(in general, infinitary) sentence; these results are 
reproduced as Theorems \ref{non-expressibility} 
and~\ref{expressibility} below. 
Nevertheless, in certain cases $\vartheta(\varphi)$ are
first-order expressible. In this note, we provide several
(syntactical as well as semantical) characterizations of 
the case when $\vartheta(\varphi)$ is equivalent to 
a~sentence in $\mathcal L_{\kappa,\kappa}$ and $\kappa$ 
is either~$\omega$ or a~certain large (e.g., compact) 
cardinal.

We start with some obvious observations, which confirm, 
in particular, that $\vartheta$~behaves like 
an S4~possibility operator.

\begin{prop}\label{connectives}
For any sentences $\varphi$, $\psi$, and $\varphi_i$, 
$i\in I$, in a~model-theoretic language~$\mathcal L$ 
involving the syntactic operations under consideration, 
we have:
\begin{itemize}
\setlength\itemsep{-0.3em}
\item[(i)]
$\vartheta(\top)$ is equivalent to $\top$, and 
$\vartheta(\bot)$ is equivalent to $\bot$; 
\item[(ii)]
$\varphi$ implies $\vartheta(\varphi)$; 
\item[(iii)]
$\vartheta(\varphi)$ is 
equivalent to $\vartheta(\vartheta(\varphi))$;
\item[(iv)]
$\varphi\to\psi$ implies 
$\vartheta(\varphi)\to\vartheta(\psi)$;
\item[(v)]
$\vartheta(\bigwedge_{i\in I}\varphi_i)$ implies
$\bigwedge_{i\in I}\vartheta(\varphi_i)$, and
$\vartheta(\bigvee_{\!i\in I}\varphi_i)$ is 
equivalent to $\bigvee_{\!i\in I}\vartheta(\varphi_i)$;
\item[(vi)]
$\neg\,\vartheta(\varphi)$ implies $\neg\,\varphi$, 
and $\neg\,\varphi$ implies $\vartheta(\neg\,\varphi)$;
\item[(vii)]
$\vartheta(\neg\,\vartheta(\varphi))$ implies 
$\vartheta(\neg\,\varphi)$. 
\end{itemize}
\end{prop}

\begin{proof}
Items (i)--(v) are immediate; (vi)~follows from~(ii);
and (vii)~from (iv) and~(vi). 
\end{proof}

\begin{corollary}\label{extensions}
For every sentence~$\varphi$, 
the sentence~$\vartheta(\varphi)$ is preserved 
under extensions of models. A~fortiori, it is 
preserved under elementary extensions, 
unions of increasing chains of models; 
in purely predicate signatures: under 
direct unions, direct products and powers; etc.
\end{corollary}

\begin{proof}
This follows from item~(iii) of 
Proposition~\ref{connectives}.
\end{proof}

Given cardinals $\kappa,\mu$ with $\mu\ge\kappa$, recall 
that $\kappa$~is $\mu$-\emph{compact} iff the language 
$\mathcal L_{\kappa,\kappa}$ satisfies the
$(\mu,\kappa)$-compactness, i.e., any theory in 
$\mathcal L_{\kappa,\kappa}$ of cardinality~$\le\mu$ 
has a~model whenever each its subtheory 
of cardinality~$<\kappa$ has a~model. 
A~cardinal~$\kappa$ is \emph{weakly compact} iff it is
$\kappa$-compact, and \emph{strongly compact} iff it is
$\mu$-compact for all $\mu\ge\kappa$ (e.g.,~$\omega$~is
strongly compact). By a~compact cardinal~$\kappa$ we shall 
mean strongly compact~$\kappa$, and by an inaccessible, 
a~strongly inaccessible, i.e., a~regular~$\kappa$ 
such that $2^\lambda<\kappa$ for all $\lambda<\kappa$. 
For more on these and other large cardinals and their
connections with infinitary languages, we refer 
the reader to \cite{Drake},~\cite{Kanamori}.

\begin{corollary}\label{existential} 
Let $\kappa$ be $\omega$ or, more generally, a~compact 
cardinal, and $\varphi$ a~sentence 
in~$\mathcal L_{\kappa,\kappa}$. 
The following are equivalent:
\begin{itemize}
\setlength\itemsep{-0.3em}
\item[(i)]
$\vartheta(\varphi)$ is equivalent to a~sentence 
in $\mathcal L_{\kappa,\kappa}$;
\item[(ii)]
$\vartheta(\varphi)$ is equivalent to a~sentence  
in $\varSigma^{0}_1(\mathcal L_{\kappa,\kappa})$, i.e., 
an existential sentence in~$\mathcal L_{\kappa,\kappa}$. 
\end{itemize}
\end{corollary}

\begin{proof}
(i)$\to$(ii). 
By Corollary~\ref{extensions} since sentences 
in~$\mathcal L_{\kappa,\kappa}$ that are preserved 
under extensions are exactly existential ones
(for $\kappa=\omega$ see~\cite{Chang Keisler}; for 
compact $\kappa>\omega$, modify the same argument). 
\end{proof}

Two following results on expressibility of 
$\vartheta(\varphi)$ were essentially obtained 
in~\cite{Saveliev Shapirovsky 2018}. They show that 
the expressibility generally does not hold in first-order 
languages but is achieved in appropriate second-order ones.  


\begin{theorem}\label{non-expressibility}
For any signature~$\tau$ having a~predicate symbol~$R$ 
of arity at least~$2$, there exists a~sentence~$\varphi$ 
in $\mathcal L_{\omega,\omega}$ such that 
$\vartheta(\varphi)$ is not equivalent to 
any sentence in~$\mathcal L_{\infty,\omega}$.
\end{theorem}

\begin{proof}
Suppose w.l.g.~that $R$~is a~binary predicate symbol 
(otherwise imitate it by using a~predicate symbol of 
a~bigger arity with fixed other arguments). Let $\varphi$ 
be an obvious $\mathcal L_{\omega,\omega}$-sentence 
saying that there exists no $R$-minimal element. Then 
$\neg\,\vartheta(\varphi)$ says that each submodel has 
a~$R$-minimal element. Note that, if $\tau$~has no 
functional symbols, $\neg\,\vartheta(\varphi)$ says that 
$R$~is well-founded. As well-known, the latter property 
is not expressible in~$\mathcal L_{\infty,\omega}$ 
(moreover, it is not RPC in~$\mathcal L_{\infty,\omega}$; 
see, e.g., \cite{Barwise Feferman}, Chapter~9, 
Theorem~3.2.20), which proves the theorem for such~$\tau$. 
In the general case, we argue as follows. 

Toward a~contradiction, assume that there is~$\kappa$ 
such that $\neg\,\vartheta(\varphi)$ is equivalent to 
some $\psi\in\mathcal L_{\kappa,\omega}$.
It follows from Karp's theorem (see, e.g.,
\cite{Rosenstein}, Theorem~14.29) that there 
are models $\mathfrak A_0$ and $\mathfrak B_0$ of the 
subsignature $\tau_0=\{R\}$ such that $\mathfrak A_0$~is 
isomorphic to an ordinal while $\mathfrak B_0$~is not, and 
$
\mathfrak A_0\equiv_{\mathcal L_{\kappa,\omega}}%
\!\mathfrak B_0.
$
Expand $\mathfrak A_0$ and $\mathfrak B_0$ respectively 
to models $\mathfrak A$ and $\mathfrak B$ of~$\tau$ by 
interpreting each predicate symbol other than~$R$ by 
the empty set, each functional symbol of positive arity 
by the projection onto the first argument, and each 
constant symbol by the $R$-last element of the model 
(which w.l.g.~can be assumed to exist). It is easy 
to see that in both $\mathfrak A$ and~$\mathfrak B$ any 
formula of~$\tau$ is equivalent to a~formula of~$\tau_0$; 
so we still have
$
\mathfrak A\equiv_{\mathcal L_{\kappa,\omega}}%
\!\mathfrak B.
$
On the other hand, in both models every subset forms 
a~submodel whenever it contains the $R$-last element of 
the whole model, whence it easily follows that 
$\mathfrak A\vDash\psi$ and $\mathfrak B\vDash\neg\,\psi$. 
A~contradiction. 
\end{proof}


Given a~model-theoretic language~$\mathcal L$, 
let $\mathcal L^{\alpha}$ denote the $\alpha$th-order 
extension of~$\mathcal L$. A~formula of $\mathcal L^{2}$ 
is \emph{monadic} iff it involves only unary predicate
variables, and \emph{existential second-order}, 
respectively, \emph{universal second-order} iff 
it involves only existential, respectively, universal
quantifiers over second-order variables preceding 
a~first-order formula (with arbitrary quantifiers).
The monadic fragment of $\mathcal L^{2}$ consists of 
its monadic formulas; similarly for the existential 
and universal fragments of the language, which will be 
denoted by $\varSigma^{1}_{1}(\mathcal L^{2})$ and
$\varPi^{1}_{1}(\mathcal L^{2})$, respectively.

\begin{theorem}\label{expressibility}
Let $\kappa$ be a~regular cardinal and $\varphi$
a~(first-order) sentence of $\mathcal L_{\kappa,\omega}$ 
in a~signature~$\tau$ with~$<\kappa$~functional 
(including constant) symbols and arbitrarily many 
predicate symbols. 
Then $\vartheta(\varphi)$ is equivalent to a~monadic 
existential formula in $\mathcal L^{2}_{\kappa,\omega}$.
Moreover, the following languages are closed 
under~$\vartheta$: 
\begin{itemize}
\setlength\itemsep{-0.3em}
\item[(i)]
the monadic fragment of $\mathcal L^{2}_{\kappa,\lambda}$
for any $\lambda\le\kappa$;
\item[(ii)]
the existential fragment of $\mathcal L^{2}_{\kappa,\lambda}$
for any $\lambda\le\kappa$;
\item[(iii)]
$\mathcal L^{\alpha}_{\kappa,\lambda}$ 
for any $\lambda\le\kappa$ and $\alpha\ge2$.
\end{itemize}
\end{theorem}

\begin{proof}
If $X$~is a~second-order unary predicate variable, 
for each functional symbol $F$ in~$\tau$ let $\psi_F$
be an $\mathcal L^{2}_{\omega,\omega}$-formula stating 
that $X$~is closed under~$F$, and let $\psi(X)$ be 
the $\mathcal L^{2}_{\kappa,\omega}$-formula 
$$
\exists x\;
\bigl(X(x)\wedge\bigwedge\{\psi_F(X):F
\text{ is a~functional symbol in }\tau\}
\bigr)
$$
stating that $X$~forms a~submodel. Then 
$\vartheta(\varphi)$~is equivalent to the sentence 
$$
\exists X\;\bigl(\psi(X)\wedge\varphi^X\bigr)
$$ 
where $\varphi^X$~is the relativization of $\varphi$
to~$X$. 
%
\end{proof}

As usual, a~filter~$D$ is $\kappa$-\emph{complete} 
iff $\bigcap E\in D$ for all $E\in P_\kappa(D)$, 
where $P_\kappa(A)$ denotes the set of all subsets 
of~$A$ which have cardinality~$<\kappa$.

\begin{corollary}\label{ultraproducts}
Let $\kappa$ be $\omega$ or, more generally, 
a~compact cardinal, and $\varphi$ a~sentence 
in $\mathcal L_{\kappa,\kappa}$ in a~signature
with~$<\kappa$~functional (including constant) symbols. 
Then $\vartheta(\varphi)$ is preserved under 
ultraproducts by $\kappa$-complete ultrafilters.
Moreover, this remains true for sentences~$\varphi$ 
in $\varSigma^{1}_1(\mathcal L^{2}_{\kappa,\kappa})$.
\end{corollary}

\begin{proof}
By Theorem~\ref{expressibility}, for such a~$\varphi$ 
the statement $\vartheta(\varphi)$ is equivalent to 
a~$\varSigma^{1}_1
(\mathcal L^{2}_{\kappa,\kappa})$-sentence, 
therefore, it is preserved under ultraproducts 
(for $\kappa=\omega$ see, e.g., \cite{Chang Keisler}, 
Corollary~4.1.14; for compact~$\kappa$ modify the same 
argument). 
\end{proof}

The next result on non-expressibility of $\vartheta(\varphi)$
shows that the restriction on the number of functional 
symbols in Theorem~\ref{expressibility} is optimal.

\begin{theorem}\label{total non-expressibility}
For any signature~$\tau$ having~$\ge\kappa$~functional 
(e.g., constant) symbols, there exists a~sentence~$\varphi$ 
of $\mathcal L_{\omega,\omega}$ 
(in fact, in the empty signature) such that 
$\vartheta(\varphi)$ is not equivalent to 
any sentence in $\mathcal L^{\alpha}_{\kappa,\kappa}$ for 
all~$\alpha$, and moreover, in every language~$\mathcal L$
whose formulas~$\psi$ have cardinality $|\psi|<\kappa$. 
\end{theorem}

\begin{proof}
Clearly, it suffices to consider only a~signature~$\tau$ 
consisting of $\kappa$~constant symbols, say, $c_\alpha$, 
$\alpha<\kappa$. Let $\varphi$~be the sentence 
$\forall x\,\forall y\:x=y$; then $\vartheta(\varphi)$ 
states the existence of a~single-point submodel. 
Toward a~contradiction, assume that such an~$\mathcal L$
has some~$\psi$ equivalent to $\vartheta(\varphi)$. 
Let two models $\mathfrak A$ and~$\mathfrak B$ in~$\tau$
have the same two-point universe $\{a,b\}$, and let 
for all $\alpha<\kappa$, 
\begin{align*}
c^{\mathfrak A}_\alpha:=a,
\quad\text{and}\quad
c^{\mathfrak B}_\alpha:=
\left\{
\begin{array}{lll}
a&\text{if }c_\alpha\text{ occurs in }\psi,
\\
b&\text{otherwise.}
\end{array}
\right.
\end{align*}
Since $|\psi|<\kappa$, there exists $\alpha<\kappa$ 
such that $c^{\mathfrak B}_\alpha=b$. So we have: 
$\mathfrak A\vDash\psi$ iff $\mathfrak B\vDash\psi$ 
(as $\mathfrak A$ and $\mathfrak B$ satisfy the same 
formulas involving only symbols from~$\psi$), 
however, $\mathfrak A\vDash\vartheta(\varphi)$ 
and $\mathfrak B\vDash\neg\,\vartheta(\varphi)$ (as 
the singleton~$\{a\}$ forms a~submodel of~$\mathfrak A$ 
while $\mathfrak B$~has no single-point submodels). 
\end{proof}


Let $\vartheta_{\le\lambda}(\varphi)$ denote that 
$\varphi$~is satisfied in a~submodel generated by 
a~set of cardinality~$\le\lambda$. Obviously, 
$\vartheta_{\le\lambda}(\varphi)$ implies
$\vartheta(\varphi)$. We are going to show that 
$\vartheta_{\le\lambda}(\varphi)$ is an existential 
sentence in an appropriate first-order language. 
To simplify some formulations, we shall consider 
\emph{partial models} in which their operations 
can be only partial. An \emph{atomic diagram} of 
a~partial model~$\mathfrak A$ is defined in the same 
way as for usual models with total operations, i.e., 
it consists of all true in~$\mathfrak A$ atomic and 
negated atomic sentences of the language expanded by 
constant symbols for all elements of~$\mathfrak A$. 

\begin{lemma}\label{inaccessible}
Let $\kappa$ be $\omega$ or, more generally, 
an inaccessible cardinal, $\lambda<\kappa$, and 
$\varphi$ a~sentence in $\mathcal L_{\kappa,\kappa}$ 
in a~signature~$\tau$ with~$<\kappa$~functional 
(including constant) symbols. Then 
$\vartheta_{\le\lambda}(\varphi)$ is equivalent to 
an existential sentence in $\mathcal L_{\kappa,\kappa}$.
\end{lemma}

\begin{proof} 
Let us first consider signatures~$\tau$ without 
functional symbols. Then $\vartheta_{\le\lambda}(\varphi)$ 
is clearly equivalent to the first-order sentence 
$$
\Bigexists_{\alpha<\lambda}\,x_\alpha\;
\varphi^{\,\{x_\alpha:\,\alpha<\lambda\}} 
$$ 
where $\varphi^{\{x_\alpha:\,\alpha<\lambda\}}$ 
is the relativization of $\varphi$ to the set of
(first-order) variables $x_\alpha$, $\alpha<\lambda$,
which do not occur in~$\varphi$. Let us verify that 
the relativization is equivalent to an open formula 
in $\mathcal L_{\kappa,\kappa}(\tau)$ 
(with parameters $x_\alpha$, $\alpha<\lambda$); 
it will clearly follow that 
$
\bigexists_{\alpha<\lambda}\,x_\alpha\:
\varphi^{\,\{x_\alpha:\,\alpha<\lambda\}} 
$ 
is equivalent to 
a~$\varSigma^{0}_1(\mathcal L_{\kappa,\kappa})$-sentence.

Indeed, $\varphi^{\{x_\alpha:\,\alpha<\lambda\}}$ is 
obtained from~$\varphi$ by successively replacing each
subformula $\bigexists_{\beta<\gamma}\,y_\beta\:\psi$
with the $\mathcal L_{\kappa,\kappa}$-formula
$$
\Bigexists_{\beta<\gamma}\,y_\beta\;
\bigl(\psi\wedge
\bigwedge\nolimits_{\beta<\gamma}
\bigvee\nolimits_{\!\!\!\alpha<\lambda}
y_\beta=x_\alpha\bigr).
$$
The latter formula is equivalent to the formula
$$
\Bigexists_{\beta<\gamma}\,y_\beta\;
\bigl(\psi\wedge
\bigvee\nolimits_{\!\!\!f\in\lambda^\gamma}
\bigwedge\nolimits_{\beta<\gamma}
y_\beta=x_{f(\beta)}\bigr),
$$
which is still in $\mathcal L_{\kappa,\kappa}$ 
since $|\lambda^\gamma|<\kappa$ due to the condition
that $\kappa$~is inaccessible, and furthermore, 
to the open formula
$$
\bigvee\nolimits_{\!\!\!f\in\lambda^\gamma}
\psi(y_\beta/x_{f(\beta)})_{\beta<\gamma}
$$
where $\psi(y_\beta/x_{f(\beta)})_{\beta<\gamma}$ 
is obtained from~$\psi$ by substituting each 
variable~$y_\beta$ with the variable~$x_{f(\beta)}$. 
This eliminates all quantifiers in all subformulas 
of~$\varphi^{\,\{x_\alpha:\,\alpha<\lambda\}}$, 
as required. 

In the general case, the construction is slightly more 
complex. Let $\tau'$~expand $\tau$ by $\lambda$~new 
constant symbols~$c_\alpha$, $\alpha<\lambda$. 
We still have $|\tau'|<\kappa$. 
Hence, since $\kappa$~is inaccessible, 
there exist only~$<\kappa$~pairwise non-isomorphic
partial models in~$\tau'$ satisfying~$\varphi$ with 
the universe constisting of an interpretion of all 
closed terms; say, $\mathfrak B_\beta$, $\beta<\mu$, 
for some $\mu<\kappa$. Note that, though such partial 
models may have size~$>\lambda$ (they interpret not only 
the~$c_\alpha$ but all terms constructed from them), 
all they have size~$\le\nu$ for some fixed $\nu$ 
with $\lambda\le\nu<\kappa$.
For any $\beta<\mu$, let $\varDelta_\beta$ be 
the atomic diagram of~$\mathfrak B_\beta$, and 
$\psi_\beta$~its conjunction $\bigwedge\varDelta_\beta$, 
which is still in $\mathcal L_{\kappa,\kappa}$ 
as $|\varDelta_\beta|<\kappa$. Let $x_{\alpha}$,
$\alpha<\lambda$, be variables not occurring
in~$\varDelta_\beta$, and let $\varphi_\beta$~be 
the formula $\psi_\beta(c_\alpha/x_\alpha)$ obtained 
from $\psi_\beta$ by replacing each constant 
symbol~$c_\alpha$ with the variable~$x_{\alpha}$. 
Then $\varphi_\beta$~is an open formula in~$\tau$, and 
the $\varSigma^{0}_1(\mathcal L_{\kappa,\kappa})$-sentence 
$$
\Bigexists_{\alpha<\lambda}\,x_\alpha\;
\varphi_\beta(x_\alpha)_{\alpha<\lambda}
$$ 
characterizes the partial model~$\mathfrak B_\beta$ 
up to isomorphism. It follows that 
$\vartheta_{\le\lambda}(\varphi)$ is equivalent to the 
$\varSigma^{0}_1(\mathcal L_{\kappa,\kappa})$-sentence
$$
\Bigexists_{\alpha<\lambda}\,x_\alpha\;
\bigvee\nolimits_{\!\!\!\beta<\mu}
\varphi_\beta(x_\alpha)_{\alpha<\lambda}.
$$
This completes the proof.
\end{proof}

\begin{remark}
The argument shows that, whenever $\kappa$~is
an inaccessible cardinal~$>\omega$, then moreover, 
$\vartheta_{\le\lambda}(\varphi)$ is equivalent to an
existential sentence in $\mathcal L_{\kappa,\lambda^+}\,$.
Also we can see that Lemma~\ref{inaccessible} remains 
true for signatures with ${<}\kappa$-ary symbols.
For $\tau$ with~$\ge\kappa$~functional symbols, 
even $\vartheta_{\le1}(\varphi)$ is non-expressible 
in any language with formulas of size~$<\kappa$, 
by the proof of Theorem~\ref{total non-expressibility}. 
\end{remark}


A~{\it fragment\/} of a~model~$\mathfrak A$ is an its 
partial submodel, i.e., a~subset of the universe of
$\mathfrak A$ together with the inherited structure. 
Thus for models in signatures without functional symbols, 
fragments are just submodels; while for models in 
signatures with functional symbols, operations on 
fragments can be partial. A~fragment can be considered 
as a~submodel of the corresponding model in the purely 
predicate language obtained from the original language 
by replacing each functional symbol of arity~$\ge1$ 
with a~predicate symbol having the same interpretation. 
Clearly, for any fragment there exists the smallest 
submodel including it, the submodel {\it generated\/} 
by the fragment.

Let $\mathfrak A$ be a~model and $I$~an ideal over~$A$ 
(the universe of~$\mathfrak A$) with $\bigcup I=A$. 
We shall say that the system $(\mathfrak B_i)_{i\in I}$ 
of models (in the same signature) is 
\emph{coherent in}~$\mathfrak A$ iff for every $i\in I$, 
the set~$i$ is included into~$B_i$ (the universe 
of~$\mathfrak B_i$) and the fragments of $\mathfrak A$ 
and of $\mathfrak B_i$ given by~$i$ coincide. 

As usual, an \emph{upper cone} of a~partially ordered set
$(P,\le)$ is an $C\subseteq P$ which is upward closed, i.e., 
such that $b\in C$ whenever $a\le b$ for some $a\in C$. 
Clearly, the set of upper cones of~$P$ generates a~filter 
over~$P$ whenever $P$~is directed. 

\begin{lemma}\label{coherent}
Let $(\mathfrak B_i)_{i\in I}$ be coherent 
in~$\mathfrak A$ and $D$~a~filter over~$I$ extending 
the filter generated by upper cones of $(I,\subseteq)$. 
Then $\mathfrak A$ isomorphically embeds into 
$\mathfrak B:=\prod_D\mathfrak B_i$, the product 
of the models~$\mathfrak B_i$ reduced by~$D$. 
\end{lemma}

\begin{proof}
For each $i\in I$ we fix some $b_i\in B_i$, and 
for each $a\in A$, let $c_a\in\prod_{i\in I}B_i$ 
be the function defined by letting for all $i\in I$,
\begin{align*}
c_a(i):=
\left\{
\begin{array}{lll}
a&\text{if }a\in i,
\\
b_i&\text{otherwise.}
\end{array}
\right.
\end{align*}
Now define $f:A\to B$ by letting for all $a\in A$,
$$
f(a):=[c_a]_D,
$$
and check that $f$~is an isomorphic embedding 
of $\mathfrak A$ into~$\mathfrak B$. 

Let $R$ be an $n$-ary predicate symbol in our signature. 
We must check that for all $a_0,\ldots,a_{n-1}$ in~$A$, 
$$
R^\mathfrak A(a_0,\ldots,a_{n-1})
\;\text{ iff }\;
R^\mathfrak B(f(a_0),\ldots,f(a_{n-1})).
$$
Since $\bigcup I=A$, for any $k<n$ there is $i_k\in I$ 
with $a_k\in B_{i_k}$, and since $I$~is an ideal, 
$\bigcup_{k<n}i_k\in I$. 
Moreover, since $i\subseteq\mathfrak B_i$ for all $i\in I$, 
whenever $\bigcup_{k<n}i_k\subseteq i$ then 
$\{a_k\}_{k<n}\subseteq B_i$ and $a_k=c_{a_k}(i)$, and so, 
since the fragments of $\mathfrak A$ and $\mathfrak B_i$ 
given by~$i$ coincide, $R^\mathfrak A(a_0,\ldots,a_{n-1})$ 
is equivalent to 
$R^{\mathfrak B_i}(c_{a_0}(i),\ldots,c_{a_{n-1}}(i))$. 
Thus we have:
\begin{align*}
R^\mathfrak A(a_0,\ldots,a_{n-1})
&\;\text{ iff }\;
\bigl\{i\in I:
R^{\mathfrak B_i}(c_{a_0}(i),\ldots,c_{a_{n-1}}(i))
\bigr\}
\text{ is an upper cone of }I
\\
&\;\text{ iff }\;
\bigl\{i\in I:
R^{\mathfrak B_i}(c_{a_0}(i),\ldots,c_{a_{n-1}}(i))
\bigr\}
\in D
\end{align*}
where one implication in the second equivalence holds 
since $D$~extends the filter generated by upper cones of~$I$ 
while the converse implication holds since the property 
inherits upward. Finally, the latter assertion is equivalent 
to $R^{\mathfrak B}([c_{a_0}]_D,\ldots,[c_{a_{n-1}}]_D)$ 
by definition of reduced products, and thus to 
$R^{\mathfrak B}(f(a_0),\ldots,f(a_{n-1}))$, as required. 

Let now $F$ be an $n$-ary functional symbol in the signature. 
We must check that for all $a_0,\ldots,a_{n-1}$ in~$A$, 
$$
f(F^\mathfrak A(a_0,\ldots,a_{n-1}))=
F^\mathfrak B(f(a_0),\ldots,f(a_{n-1})).
$$
Indeed, 
$$
f(F^\mathfrak A(a_0,\ldots,a_{n-1}))=
[c_{F^\mathfrak A(a_0,\ldots,a_{n-1})}]_D
$$
while 
\begin{gather*}
F^\mathfrak B(f(a_0),\ldots,f(a_{n-1}))=
F^\mathfrak B([c_{a_0}]_D,\ldots,[c_{a_{n-1}}]_D)=
[b]_D
\\
\text{ where }
b(i)=
F^{\mathfrak B_i}(c_{a_0}(i),\ldots,c_{a_{n-1}}(i)).
\end{gather*}
Again, if for $k<n$, $i_k\in I$ is such that 
$a_k\in B_{i_k}$, and also $i_{n}\in I$ is such that 
$F^\mathfrak A(a_0,\ldots,a_{n-1})\in B_{i_n}$,
whenever $\bigcup_{k\le n}i_k\subseteq i$ then 
$
c_{F^\mathfrak A(a_0,\ldots,a_{n-1})}(i)=
F^\mathfrak A(a_0,\ldots,a_{n-1})=
F^{\mathfrak B_i}(a_0,\ldots,a_{n-1})
$
and also 
$
F^{\mathfrak B_i}(c_{a_0}(i),\ldots,c_{a_{n-1}}(i))=
F^{\mathfrak B_i}(a_0,\ldots,a_{n-1}).
$
It follows 
$
[c_{F^\mathfrak A(a_0,\ldots,a_{n-1})}]_D=[b]_D,
$
as required. 

The proof is complete.
\end{proof}

\begin{remark}
In general, even if $D$~is an ultrafilter and all 
the~$\mathfrak B_i$ are submodels of~$\mathfrak A$, the 
embedding is not elementary, and moreover, $\mathfrak A$ 
and $\mathfrak B$ are not elementarily equivalent, 
even in the sense of $\mathcal L_{\omega,\omega}$. 
E.g., let $\mathfrak A=(\omega,<)$ and 
$\mathfrak B_i=(i,<)$ for all finite $i\subseteq\omega$. 
Then if $\varphi$ is $\exists x\,\forall y\,\neg\,(x<y)$, 
we have: $\mathfrak A\vDash\varphi$, but for all~$i$, 
$\mathfrak B_i\vDash\neg\,\varphi$, and hence, 
$\mathfrak B\vDash\neg\,\varphi$.
\end{remark} 

\begin{remark}
If the ideal~$I$~is $\kappa$-complete, i.e., 
$\bigcup E\in I$ for all $E\in P_\kappa(I)$,
then Lemma~\ref{coherent} remains true even for signatures
involving ${<}\kappa$-ary symbols. Let us point out also 
that whenever $I$~is $\kappa$-complete then so is 
the filter~$D_I$ over~$I$ generated by upper cones 
in $(I,\subseteq)$ (but of course not any filter~$D$
extending~$D_I$), and that in the case $I=P_\kappa(A)$, 
$D_I$~is the least $\kappa$-complete \emph{fine} filter 
over $P_\kappa(A)$.
\end{remark}


The theorem below is the main result of this note; it 
extends Corollary~\ref{existential} by providing new
characterizations\;--\;syntactical in item~(iii) and 
semantical in items (iv) and~(v)\;--\;of the case when
$\vartheta(\varphi)$ is equivalent to a~first-order formula.

\begin{theorem}\label{main}
Let $\kappa$ be $\omega$ or, more generally, 
a~compact cardinal and $\varphi$ a~sentence 
in the language $\mathcal L_{\kappa,\kappa}$ 
in a~signature~$\tau$ with~$<\kappa$~functional 
(including constant) symbols. 
The following are equivalent:
\begin{itemize}
\setlength\itemsep{-0.3em}
\item[(i)]
$\vartheta(\varphi)$ is equivalent to 
a~$\varSigma^{0}_1(\mathcal L_{\kappa,\kappa})$-sentence;
\item[(ii)]
$\vartheta(\varphi)$ is equivalent to 
an $\mathcal L_{\kappa,\kappa}$-sentence;
\item[(iii)]
$\vartheta(\varphi)$ is equivalent to 
a~$\varPi^{1}_1(\mathcal L^{2}_{\kappa,\kappa})$-sentence;
\item[(iv)]
any model satisfying $\varphi$ has a~fragment of 
cardinality~$<\kappa$ such that each model having 
the fragment satisfies~$\vartheta(\varphi)$;
\item[(v)]
there exists $\lambda<\kappa$ such that 
any model satisfying $\varphi$ has a~fragment 
of cardinality~$\le\lambda$ and such that each model 
having the fragment satisfies~$\vartheta(\varphi)$.
\end{itemize}
\end{theorem}

\begin{proof}
(i)$\to$(ii) and (ii)$\to$(iii). 
Trivial.

(iii)$\to$(iv).
Assume that (iv)~does not hold. Then there is $\mathfrak A$
such that $\mathfrak A\vDash\varphi$, and for every set 
$i\subseteq A$ of size~$<\kappa$, there exists 
a~model~$\mathfrak B_i$ such that $i\subseteq B_i$, 
the fragments of $\mathfrak A$ and of $\mathfrak B_i$ 
given by~$i$ coincide, and $\mathfrak B_i$~has 
no submodels satisfying~$\varphi$, 
thus $\mathfrak B_i\vDash\neg\,\vartheta(\varphi)$. 
Let $\mathfrak B:=\prod_D\mathfrak B_i$ where $D$~is 
a~$\kappa$-complete ultrafilter over $P_\kappa(A)$ which 
is fine, i.e.,~extends the filter generated by the sets 
$\{i\in P_\kappa(A):a\in i\}$ for all $a\in A$ (recall 
that the existence of such an ultrafilter follows from
the compactness of~$\kappa$; see, e.g., \cite{Kanamori}, 
Corollary~22.18). By Lemma~\ref{coherent}, 
$\mathfrak A$~isomorphically embeds into~$\mathfrak B$;
therefore, $\mathfrak B\vDash\vartheta(\varphi)$. 
Let us show that (iii)~fails. 

Indeed, if $\vartheta(\varphi)$~is equivalent to 
a~$\varPi^{1}_1(\mathcal L^{2}_{\kappa,\kappa})$-formula, 
then $\neg\,\vartheta(\varphi)$ is equivalent to 
a~$\varSigma^{1}_1(\mathcal L^{2}_{\kappa,\kappa})$-formula,
and hence, is preserved under ultraproducts by 
$\kappa$-complete ultrafilters (as was pointed out in 
the proof of Corollary~\ref{ultraproducts}), whence 
we get also $\mathfrak B\vDash\neg\,\vartheta(\varphi)$;
a~contradiction. 

(iv)$\to$(v). 
Assume that (v)~does not hold. Let us again use 
an ultraproduct argument: for any $\alpha<\kappa$ pick
a~model~$\mathfrak A_\alpha$ which satisfies~$\varphi$ 
and does not have fragments of size~$\le\lambda$ 
generating submodels satisfying~$\varphi$, 
pick any $\kappa$-complete ultrafilter~$D$ over~$\kappa$, 
and consider $\mathfrak A:=\prod_D\mathfrak A_\alpha$. 
Clearly, $\mathfrak A$~satisfies~$\varphi$. Let us show 
that $\mathfrak A$~does not have fragments of size~$<\kappa$
generating submodels satisfying~$\varphi$, thus proving
that (iv)~fails. 

Indeed, if there is $\lambda<\kappa$ such that 
$\mathfrak A$ has some $\lambda$-generated submodel 
satisfying~$\varphi$, i.e., 
$\mathfrak A\vDash\vartheta_{\le\lambda}(\varphi)$, 
then this fact is expressed by a~first-order (and even 
existential) sentence by Lemma~\ref{inaccessible},
and hence, should hold in $\mathfrak A_\alpha$ 
for $D$-almost all~$\alpha$, which is, however, not true. 

(v)$\to$(i). 
Assume (v). Then $\vartheta(\varphi)$ is equivalent to 
$\vartheta_{\le\lambda}(\varphi)$, which is equivalent to 
a~$\varSigma^{0}_1(\mathcal L_{\kappa,\kappa})$-sentence
by Lemma~\ref{inaccessible}, thus proving~(i).

The theorem is proved.
\end{proof}


\begin{corollary}\label{predicate}
Let $\kappa$ be $\omega$ or, more generally, 
a~compact cardinal and $\varphi$ a~sentence in 
$\mathcal L_{\kappa,\kappa}$ in a~signature without 
functional symbols. The following are equivalent:
\begin{itemize}
\setlength\itemsep{-0.3em}
\item[(i)]
$\vartheta(\varphi)$ is equivalent to 
a~sentence in $\mathcal L_{\kappa,\kappa}$
(or $\varSigma^{0}_1(\mathcal L_{\kappa,\kappa})$, or 
$\varPi^{1}_1(\mathcal L^{2}_{\kappa,\kappa})$);
\item[(ii)]
any model satisfying $\varphi$ has a~submodel 
of cardinality~$<\kappa$ satisfying~$\varphi$ 
(or $\vartheta(\varphi)$); 
\item[(iii)]
there exists $\lambda<\kappa$ such that 
any model satisfying~$\varphi$ has a~submodel 
of cardinality~$\le\lambda$ satisfying~$\varphi$ 
(or $\vartheta(\varphi)$). 
\end{itemize}
\end{corollary}

\begin{proof} 
As in such signatures the notions of fragments and 
submodels coincide, this follows from Theorem~\ref{main}. 
\end{proof}


\begin{example}
Recall that models in signatures having only
unary functional symbols are called \emph{unoids}, 
and if such a~symbol is unique, \emph{unars}.

Let a~signature~$\tau$ consist of a~single 
unary functional symbol~$F$. Let $\varphi$~be 
the $\mathcal L_{\omega,\omega}$-sentence 
$\exists x\:F(x)\ne x$ in~$\tau$. Then 
$\vartheta(\varphi)$~is equivalent to $\varphi$~itself. 
Clearly, the cardinality $\lambda<\omega$ of a~finite 
fragment determining satisfiability of $\vartheta(\varphi)$
in submodels extending~it, stated in 
Theorem~\ref{main}\,(v), is~$1$; note that there are 
models of~$\varphi$ without finite submodels at all 
(e.g., so is the free $1$-generated unar $(\omega,S)$ 
where $S$~is the successor operation). 

More generally, let $\tau$~consist of $\kappa$~unary 
functional symbols~$F_\alpha$, $\alpha<\kappa$, and let
$\varphi$~be the $\mathcal L_{\kappa^+,\omega}$-sentence
$
\exists x
\bigvee_{\!\alpha<\kappa}F_\alpha(x)\ne x
$
in~$\tau$. Then again, $\vartheta(\varphi)$~is 
equivalent to~$\varphi$, $\lambda$~is~$1$, and 
there are models of~$\varphi$ without submodels 
of size~$<\kappa$ (e.g., any free $\tau$-unoid). 
\end{example}

\begin{example}
Let a~signature~$\tau$ consist of a~single binary 
predicate symbol~$R$. 

Let $\varphi$~be the $\mathcal L_{\omega,\omega}$-sentence 
$\exists x\,\forall y\:R(x,y)$ in~$\tau$. Then 
$\vartheta(\varphi)$ is equivalent to $\exists x\,R(x,x)$,
and $\lambda$ from Corollary~\ref{predicate}\,(iii)
is~$1$. Observe that $\varphi$~is in 
$\varSigma^{0}_2\,\setminus\,\varPi^{0}_2$. 
Indeed, assume that $\varphi$~is in~$\varPi^{0}_2$.
Recall that $\varPi^{0}_2$-formulas are characterized 
as those that are preserved under chain unions 
(see~\cite{Chang Keisler}, Theorem~3.2.2). Let 
$\mathfrak A_n$ be $(n+1,\ge)$, and let $\mathfrak A$~be 
the union of the chain of the~$\mathfrak A_n$, $n\in\omega$. 
Then $\mathfrak A_n\vDash\varphi$ for all $n\in\omega$
but $\mathfrak A\vDash\neg\,\varphi$; a~contradiction. 

Let also $\psi$~be the $\mathcal L_{\omega,\omega}$-sentence
$\forall x\,\exists y\:R(x,y)$ in~$\tau$. Then 
$\vartheta(\psi)$ is not equivalent to any 
$\mathcal L_{\omega,\omega}$-sentence. Indeed, 
the model $(\omega,<)$ satisfies $\psi$ but includes no
finite submodels satisfying~$\psi$; apply 
Corollary~\ref{predicate}\,(iii). 
Similar arguments show that 
$\psi$~is in $\varPi^{0}_2\,\setminus\,\varSigma^{0}_2$
and that $\vartheta(\neg\,\psi)$ is 
$\exists x\,\neg\,R(x,x)$ while $\vartheta(\neg\,\varphi)$ 
is not equivalent to 
any $\mathcal L_{\omega,\omega}$-sentence.
\end{example}


\begin{remark}
Although any~$\varphi$ with the first-order 
expressible $\vartheta(\varphi)$ determines the least size 
of a~small fragment from Theorem~\ref{main}, which is thus
independent of a~particular model, it does not determine,
even up to isomorphism, the fragment itself. E.g., if
$\varphi$~is any true formula then $\vartheta(\varphi)$ 
is equivalent to~$\varphi$; hence in a~purely predicate signature~$\tau$ the discussed size is~$1$ but there can 
be many non-isomorphic single-point models\;--\;arbitrarily 
many in an appropriate~$\tau$. 
\end{remark}

\begin{remark} 
Despite the low complexity of the first-order expressible 
$\vartheta(\varphi)$, the complexity of $\varphi$~itself 
can be much higher: there exist non-first-order 
expressible~$\varphi$ with first-order
expressible~$\vartheta(\varphi)$. Let $\varphi$~state 
the finiteness, i.e., let any model satisfy $\varphi$ 
iff it is finite. As well-known, $\varphi$~is not 
$\mathcal L_{\omega,\omega}$-expressible (though is 
expressible in $\mathcal L_{\omega_1,\omega}$ or 
else in the weak second-order language; see, 
e.g.,~\cite{Barwise Feferman}). However, 
in any signature without functional symbols,
$\vartheta(\varphi)$~is equivalent to any true formula. 
This example is generalized to the languages
$\mathcal L_{\kappa,\kappa}$ with arbitrarily 
large~$\kappa$. 
\end{remark}

\begin{remark}
Theorem~\ref{main} (and Corollary~\ref{predicate}) 
remains true even for signatures with ${<}\kappa$-ary 
symbols; for Lemmas \ref{inaccessible} 
and~\ref{coherent} this has been already noticed, other
arguments in the proof do not require any modifying.
\end{remark}

\begin{remark}
The results of this note admit further improvements 
and generalizations. 

First, in Theorem~\ref{main} 
(and Corollary~\ref{predicate}) it suffices to assume that 
$\kappa$~is inaccessible. Moreover, for $\kappa>\omega$, 
if a~sentence~$\varphi$ in $\mathcal L_{\kappa,\kappa}$ 
holds in a~model then it holds in its submodel of 
size~$<\kappa$ (see, e.g., \cite{Barwise Feferman},
Corollary~3.1.3), hence $\vartheta(\varphi)$ is 
equivalent to $\vartheta_{\le\lambda}(\varphi)$ 
for some $\lambda<\kappa$; then items (i)--(v) of
Theorem~\ref{main} follow due to Lemma~\ref{inaccessible} 
(even in a~nicer form with submodels instead of fragments, 
like as in Corollary~\ref{predicate}). 

Further, the obtained results can be relativized to 
theories~$T$ in a~given language. Another natural 
generalization concerns the question when $\vartheta(T)$ 
is equivalent to some theory in the same language. 
\end{remark}

\begin{question}
Characterize first-order sentences~$\varphi$ for 
which $\vartheta(\varphi)$ are first-order sentences 
in $\mathcal L_{\omega,\omega}$. 
\end{question}

\begin{acknowledgement}
I~am grateful to N.\,L.~Poliakov for discussions 
on the subject of this note and especially for his 
valuable help in handling the case of functional 
signatures in Lemma~\ref{inaccessible}. I~am indebted 
to F.\,N.~Pakhomov for his remark about the number 
of functional symbols in that lemma, which leaded me 
to Theorem~\ref{total non-expressibility}, and for 
his proposal to weaken the large cardinal property 
of~$\kappa$ to inaccessibility by using the downward 
L\"owenheim--Skolem theorem for $\mathcal L_{\kappa,\kappa}$. 
I~also express my appreciation to I.\,B.~Shapirovsky 
who read this note and made several useful comments. 
\end{acknowledgement}


\end{document}